\documentclass[a4paper]{amsart}
\usepackage{amsmath, amssymb, amsthm, mathrsfs}
\usepackage[margin=30truemm]{geometry}

\theoremstyle{plain}
\newtheorem{thm}{Theorem}[section]
\newtheorem{prop}[thm]{Proposition}
\newtheorem{lem}[thm]{Lemma}

\theoremstyle{definition}
\newtheorem{dfn}[thm]{Definition}
\newtheorem{exm}[thm]{Example}
\newtheorem{cond}[thm]{Condition}

\begin{document}

%%% Title 
\title{On the number of points with bounded dynamical canonical height}
\author[K. Takehira]{Kohei Takehira}
\address{Graduate School of Science, Tohoku University, 
	6-3, Aoba, Aramaki-aza, Aoba-ku, 
	Sendai, 980-8578, Japan}
\email{kohei.takehira.p5@dc.tohoku.ac.jp}
\date{\today}

\maketitle

%% Abstract
\begin{abstract}
  This paper discusses the number of points for which the dynamical canonical height is less than or equal to a given value.
  The height function is a fundamental and important tool in number theory to capture the ``number-theoretic complexity" of a point.
  Asymptotic formulas for the number of points in projective space below a given height have been studied by Schanuel \cite{Schanuel}, for example, and their coefficients can be written by class numbers, regulators, special values of the Dedekind zeta function, and other number theoretically interesting values.
  We consider an analogous problem for dynamical canonical height, a dynamical analogue of the height function in number theory, introduced by Call-Silverman \cite{Call-Silverman}.
  The main tool of this study is the dynamical height zeta function studied by Hsia \cite{Hsia}.
  In this paper, we give explicit formulas for the dynamical height zeta function in special cases, derive general formulas for obtaining asymptotic behavior from certain functions, and combine them to derive asymptotic behavior for the number of points with bounded dynamical canonical height.
\end{abstract}

%% Introduction 
\section{Introduction}

 Let $K$ be a global field and $M_K$ denotes the set of all places of $K$.
The (relative) height function $H_K \colon K \to \mathbb{R}$ is defined by 
\[ H_K(x) = \prod_{v \in M_K} \max \{1, |x|_v\}\]
where $|\bullet|_v$ is the normalized absolute value associated to $v \in M_K$ (cf. \cite[PART B.]{Hindry-Silverman}).
The height function is a fundamental tool in number theory, which measures ``arithmetic complexity" of $x$. 
The height function is not only technically important, but is also an interesting research subject in its own right.
For example, Schanuel\cite{Schanuel}(1964) showed the following asymptotic formula for the number of points in projective $N$ space over a number field $K$.
\begin{align*}
  N(\mathbb{P}^{N}_K, H_K, B) &= \# \{ x \in \mathbb{P}^{N}_{K} \colon H_K(x) \le B\}\\
  &= C_{K,N} B^{N+1} + 
  \begin{cases}
    O(B \log B) & \text{ if } N = [K:\mathbb{Q}] = 1\\
    O(B^{N+1-1/[K:\mathbb{Q}]}) & \text{ otherwise }
  \end{cases}
\end{align*}
where 
\[ C_{K,N} = \frac{h_K R_K}{w_K \zeta_K(N+1)} (N+1)^{r_1+r_2-1} \left( \frac{2^{r_1}(2\pi)^{r_2}}{\sqrt{|D_K|}} \right)^{N+1}\]
Here, $h_K$ denotes the class number of $K$, $R_K$ represents the regulator of $K$, $w_K$ is the count of roots of unity in $K$, $\zeta_K$ corresponds to the Dedekind zeta function of $K$, $r_1$ and $r_2$ signify the number of real and complex places of $K$ respectively, and $D_K$ stands for the discriminant of $K$.
%$h_K$ is the class number of $K$, $R_K$ is the regulator of $K$, $w_K$ is the number of roots of unity in $K$, $\zeta_K$ is the Dedekind zeta function of $K$, $r_1$ is the number of real embeddings of $K$, $r_2$ is the number of complex embeddings of $K$, and $D_K$ is the discriminant of $K$.

Call-Silverman(1993)\cite{Call-Silverman} introduced the dynamical variant of the height function. 
For a polynomial $\phi \in K[z]$, it is defined by
\[\widehat{H}_{\phi}(x) = \lim_{n \to \infty} H_K(\phi^n(x))^{1/d^n}.\]
The function $\widehat{H}_{\phi}$ serves as a well-suited function that effectively reflects the properties of dynamical systems. 
For instance, $\widehat{H}_{\phi}(x) = 1$ is equivalent to the condition that $x$ is preperiodic with respect to $\phi$, which means that $\phi^i(x)$ is periodic with respect to the iterations of $\phi$ for some $i$.

This paper studies the number of points whose dynamical canonical height is less than a given value.
For this purpose, we use the dynamical height zeta function studied by Hsia.

Hsia(1997)\cite{Hsia} considered the dynamical height zeta function
\[ Z_K(\phi, s) = \sum_{x \in K} \frac{1}{\widehat{H}_{\phi}(x)^s}, \mathrm{Re}(s) > 2 \]
when $K$ is a function field of a curve over a finite field $\mathbb{F}_q$ and $\phi$ has at worst \textit{mildly bad reduction} (cf.\cite{Hsia}).
He proved the meromorphic continuation of $Z_K(\phi, s)$ to the whole plane $\mathbb{C}$.
Also, he gave an asymptotic formula for $N(K,\widehat{H}_{\phi}, B) = \# \{ x \in K \colon \widehat{H}_{\phi}(x) \le B\}$ by using the residue of $Z_{K}(\phi, s)$ at $s = 2$.
However, the author think that his asymptotic formula contains a small error since $Z_{K}(\phi, s)$ can have infinitely many poles in the vertical line $\mathrm{Re}(s) = 2$.

The purpose of this paper is to give an argument to correct that error and to give more precise asymptotic formulas in some specific senarios.

The main results of this paper can be summarized as follows:
\begin{itemize}
  \item (Theorem \ref{main1}) Explicit formulas for $Z_K(\phi, s)$ for $\phi(z)$ belonging to certain families, when $K$ is a function field of genus 0 or 1 over $\mathbb{F}_q$.
  \item (Theorem \ref{main2}) General formulas for deriving asymptotic behavior from zeta functions.
\end{itemize}
To be more precise, the main claims are as follows:

\begin{thm}[{Theorem \ref{explicit_formula_genus_1}}]\label{main1}
  Let $K$ be a funciton field of a curve over a fintie field $\mathbb{F}_q$ of genus $0$ or $1$ and let $d \ge 2$ be an integer greater than or equal to $2$.
  Assume that $f \in \mathcal{O}_K$ satisfies $v(f) < d$ for all $v \in M_K$ and put $S = \{v \in M_K \colon v(f) > 0\}$. 
  Then, the dynamical height zeta function of $\phi(z) =z^d + f^{-1}$ is given by the next formula.
  \[Z_K(\phi, s) = q^{1-g}\frac{\zeta_K(s-1)}{\zeta_K(s)} \prod_{v \in S} \frac{u_v^{v(f)} + (q_v - 1)u_v^d - q_v u_v^{d + v(f)}}{1 - u_v^d} + \frac{c(g)}{\zeta_K(s)} \prod_{v \in S} \frac{u_v^{v(f)} - u_v^d}{1 - u_v^d}.\]
  Here, $\zeta_K$ represents the Dedekind zeta function of $K$, $q_v = q^{f_v}$ stands for the cardinality of the residue field at $v \in M_K$, $u_v = q^{-sf_v/d}$, $c(0)=0$, and $c(1) = q-1$.
\end{thm}

\begin{thm}[{Theorem \ref{GeneralAsymptotic}}]\label{main2}
  Let $S$ be a set, and let $H \colon S \to \mathbb{R}_{>0}$ be a map such that for any $B \in \mathbb{R}_{>0}$, $N(S,H,B)= \#\{x \in S \colon H(x) \le B\} < \infty$. 
  Furthermore, define $Z(S,H,s)=\sum_{x \in S} H(x)^{-s}$, and assume that it converges absolutely for complex numbers $s$ with sufficiently large real parts and that there exists $\alpha \in \mathbb{R}_{>0}$ such that $Z(S,H,s) \in \mathbb{Q}(\alpha^{-s})$.
  Let $\mathrm{Poles}$ be a set of poles of $Z(S,H,s)$ defined as follows:
  \[ \mathrm{Poles} = \left \{a \in \mathbb{C} \colon a \text{ is a pole of } f, 0 \le \mathrm{Re}(a), 0 \le \mathrm{Im}(a) < \frac{2\pi}{\log \alpha} \right \}.\]
  For a pole $a \in \mathrm{Poles}$, we write the Laurent expansion at $s = a$ as
  \[ Z(S,H,s) = \sum_{n=1}^{N_a} \frac{c_n(a)}{(\log \alpha)^n(s - a)^n} + (\text{holomorphic at }s=a).\]
  Then, the following asymptotic formula for $N(S,H,B)$ holds.
  \begin{align*}
    N(S, H, B) =  \sum_{\alpha^m \le B} \sum_{a \in \mathrm{Poles}} \alpha^{am} \sum_{n = 1}^{N_a} c_{n}(a) \frac{m^{n-1}}{(n - 1)!} + O(1).
  \end{align*}
\end{thm}

By combining these theorems, it becomes possible to explicitly derive asymptotic formulas for the number of points with bounded dynamical canonical height (Section \ref{Example}).

The structure of this paper is as follows:
\begin{itemize}
  \item In Section 2, we present the definition of the dynamical canonical height, review its fundamental properties, and introduce the dynamical height zeta function.
  \item In Section 3, we revisit the computational method for the dynamical height zeta function by Hsia. While the content of this section is entirely due to Hsia \cite{Hsia}, it is necessary for the subsequent discussions, so we provide a brief overview.
  \item In Section 4, we derive Theorem \ref{main1}, which is an explicit formula for the dynamical height zeta function.
  \item In Section 5, we prove Theorem \ref{main2}, which is a general theorem for deriving asymptotic formulas from zeta functions.
  \item Finally, in Section 6, we use these theorems to demonstrate how the asymptotic behavior of the number of points with bounded dynamical canonical height can be obtained.
\end{itemize}

\section*{Acknowledgments}

I would like to show my greatest appreciation to Professor Nobuo Tsuzuki for
his advice and helpful comments.
The author is financially supported by JSPS KAKENHI Grant Number JP 22J20227 and by the WISE Program for AI Electronics, Tohoku University.

\section{Dynamical canonical height and dynamical height zeta function}

In this section, we review the definiton and basic properties of the dynamical canonical height and the dynamical height zeta function.

\subsection{Dynamical canonical height}

In this subsection, we briefly recall the theory of the dynamical canonical height for a polynomial dynamical system.
One of basic reference of this topic is Silverman's book \cite[Chapter 3 and Section 5.9.]{Silverman}
We fix the folloing notation in this subsection.

\begin{itemize}
  \item $K$ is a global field.
  \item $M_K$ is the set of all places of $K$.
  \item $\phi \in K[z]$ is a polynomial with $\deg \phi = d \ge 2$.
  \item For $v \in M_K$, $K_v$ represents the completion of $K$ at $v$.
  \item For non-archimedean $v \in M_K$, $\mathcal{O}_v$ stands for the $v$-adic integer ring.
\end{itemize}

\begin{dfn}[{Dynamical canonical height, cf. \cite[Section 3.4.]{Silverman}}]
  \[ \widehat{H}_{\phi}(x) = \lim_{n \to \infty} H_K(\phi^n(x))^{1/d^n} \]
  \[ \widehat{h}_{\phi}(x) = \lim_{n \to \infty} \frac{1}{d^n}h_K(\phi^n(x))\]
\end{dfn}

As is often the case in many discussions in number theory, the concept of height functions over global fields can be reduced to discussions over local fields. 
Height functions over local fields are defined as follows.

\begin{dfn}[{Local canonical height, cf. \cite[Section 3.5.]{Silverman}}]
  \[ \widehat{H}_{\phi, v}(x) = \lim_{n \to \infty} \max \{1, |\phi^n(x)|_v\}\]
  \[ \widehat{\lambda}_{\phi, v}(x) = \lim_{n \to \infty} \frac{1}{d^n} \log \max \{ 1, |\phi^n(x)|_v\}\]
\end{dfn}

From height functions over local fields, height functions over global fields can be reconstructed as follows.

\begin{prop}[{cf. \cite[Theorem 3.29.]{Silverman}}]
  \[\widehat{h}_{\phi} = \sum_{v \in M_K} \widehat{\lambda}_{\phi, v}\]
\end{prop}

The behavior of the local canonical height becomes trivial for places with good reduction.
The notion of a polynomial dynamical system having good reduction is defined as follows.

\begin{dfn}[{cf. \cite[Section 2.5]{Silverman}}]
  For a global field $K$ and a polynomial $\phi \in K[z]$, $\phi$ having \textit{good reduction} at non-archimedean place $v \in M_K$ means that $\phi$ is defined with coefficients in $\mathcal{O}_v$, and its leading coefficient is a $v$-adic unit.
  In cases where the above condition does not hold, it is said that $\phi$ has \textit{bad reduction} at $v$.
\end{dfn}

For each $v \in M_K$, we define $\lambda_{v}(x) = \log \max \{1, |x|_v\}$ and $H_{v}(x) = \max \{1, |x|_v\}$. 
In this case, it's indeed true that:
\[h_K(x) = \sum_{v \in M_K} \lambda_{v}(x), H_K(x) = \prod_{v \in M_K} H_{v}(x).\]

\begin{prop}[{cf. \cite[Proposition 5.58]{Silverman}}]
  For non-archimedean $v \in M_K$, if $\phi$ has good reduction, then the following equalities hold:
  \[ \widehat{\lambda}_{\phi, v} = \lambda_{v}, \widehat{H}_{\phi, v} = H_{v}\]
\end{prop}

From the above discussion, it follows that the non-trivial cases to consider when studying $\widehat{H}_{\phi}$ are only those where $\phi$ has bad reduction at some $v \in M_K$. 
Therefore, in the subsequent discussions, the focus is primarily on $H_K$, and methods are employed to correct the differences arising from $v \in M_K$ where bad reduction occurs.

\subsection{Dynamical height zeta function}

According to the Hsia's paper\cite{Hsia}, the following definition is due to Silverman in his unpublished talk at Harbard number theory seminar in 1994.

\begin{dfn}[Dynamical height zeta function]
  Let $K$ be a global field and $\phi \in K[z]$ be a polynomial with $\deg \phi \ge 2$.
  For $s \in \mathbb{C}$ satistying $\mathrm{Re}(s) > 2$, we define the \textit{dynamical height zeta function} as
  \[Z_K(\phi, s) = \sum_{x \in K} \frac{1}{\widehat{H}_{\phi}(x)^s}.\]
\end{dfn}

Let $\mu$ be the Haar measure on the adele ring $\mathbb{A}_K$, normalized so that the fundamental domain for the action of $K$ has measure 1. 
Additionally, for each $v \in M_K$, let $\mu_v$ be the Haar measure on $K_v$, normalized so that the measure of its ring of integers is 1. 
In this context, Hsia represented $Z_K(\phi, s)$ as an integral over the adeles plus a residue term.

\begin{thm}[{\cite{Hsia}}]\label{Hsia_main}
  Let $K$ be a funciton field of genus $g$ over $\mathbb{F}_q$, and let $\phi \in K[z]$ be a polynomial with at worst mildly bad reduction.
  Then, $Z_K(\phi, s)$ has a meromorphic continuation to $\mathbb{C}$ and 
  \begin{align*}
     Z_K(\phi, s) &= \int_{\mathbb{A}_K} \widehat{H}_{\phi}(x)^{-s} d \mu(x) + R(s) \\
     &= q^{1-g} \frac{\zeta_{K}(s-1)}{\zeta_{K}(s)}\prod_{v \in S} \frac{\int_{K_v} \widehat{H}_{\phi, v} (x)^{-s} d \mu_v(x)}{\int_{K_v} H_{v} (x)^{-s} d \mu_v(x)} + R(s)
  \end{align*}
  where $R(s)$ is a meromorphic function on $\mathbb{C}$ which has the property that $R(s) \prod_{v\in S}(1-q^{-f_v s})$ is an entire function.
\end{thm}

\section{Hsia's method}\label{Hsia_method}

The discussion in this section is entirely due to Hsia \cite{Hsia}, but for the sake of future discussions, let us briefly recap the contents.
Henceforth, we fix the following notation.

\begin{itemize}
  \item $K$ is a function field over a finite field $\mathbb{F}_q$ of genus $g$.
  \item $M_K$ is the set of all places of $K$.
  \item For $v \in M_K$, $K_v$ is the completion of $K$ at $v$.
  \item $\mathcal{O}_v$ is the integer ring of $K_v$.
  \item $q_v$ is the cardinality of the residue field of $K_v$.
  \item $f_v$ is the residue degree at $v \in M_K$. (Note that $q_v = q^{f_v}.$)
\end{itemize}

In addition to the above, we assume the following condition.

\begin{cond}\label{Hsia_condition}
  Let $\phi \in K[z]$ be a polynomial over $K$, and let $S$ be a finite subset of $M_K$ such that outside of $S$, $\phi$ has good reduction. 
  For each $v \in S$, there exists a family of subsets of $K_v$, denoted as $D_v^{(1)}, \dots, D_v^{l_v}$, satisfying the following conditions:
  \begin{itemize}
    \item[$(a)$]Each $D_v^{(i)}$ takes the form of $E\setminus \bigcup_{j=1}^{m} E_j$, where $E$ is either the entire $K_v$ or a closed disk in $K_v$, and $E_j$ represents a closed disk in $K_v$.
    \item[$(b)$]On each $D_v^{(i)}$, $\widehat{\lambda}_{\phi,v}(x) - \lambda_{v}(x)$ is a constant. 
  \end{itemize}
\end{cond}

As an example of a polynomial that satisfies this condition, consider the following:

\begin{exm}\label{computation_local_height}
  Let $d \ge 2$ be an integer greater than or equal to $2$.
  Assume that $f \in K$ satisfies $v(f) < d$ for all $v \in M_K$.
  Then, $\phi(z) = z^d + f^{-1}$ satisfies the condition.
  To see this, we need to observe the non-archimedean dynamical system on $K_v$.
  Put $S = \{v \in M_K \colon \phi \text{ has bad reduction at } v \} = \{v \in M_K \colon v(f) > 0 \}$.
  Then for all $v \in S$ and $x \in K_v$,
  \begin{align*}
    v(\phi(x)) = v(x^d + f^{-1}) = 
    \begin{cases}
      d v(x) & \text{ if } v(x) < 0\\
      - v(f) & \text{ if } v(x) \ge 0
    \end{cases}.
  \end{align*}
  By induction, we obtain
  \begin{align*}
    \log |\phi^n(x)|_v =
    \begin{cases}
      d^n \log |x|_v & \text{ if } |x|_v > 1\\
      - d^{n-1} v(f) \log q_v & \text{ if } |x|_v \le 1
    \end{cases}
  \end{align*}
  By the definition of the local canonical heigt, 
  \begin{align*}
    \widehat{\lambda}_{\phi, v} (x) = 
    \begin{cases}
      \log |x|_v & \text{ if } |x|_v > 1\\
      \frac{v(f)}{d} \log q_v & \text{ if } |x|_v \le 1
    \end{cases}.
  \end{align*}
  Therefore, 
  \begin{align*}
    \widehat{\lambda}_{\phi, v} (x) - \lambda_{v} (x) = 
    \begin{cases}
      0 & \text{ if } |x|_v > 1\\
      \frac{v(f)}{d} \log q_v & \text{ if } |x|_v \le 1
    \end{cases}.
  \end{align*}
\end{exm}

Hsia proposed a method for computing the dynamical height zeta function for polynomials that satisfy Condition \ref{Hsia_condition}.
The outline of the discussion is as follows:

Firstly, assuming that Condition \ref{Hsia_condition} holds, we take a map $i \colon S \to \mathbb{Z}$ such that for each $v \in S$, $1 \le i(v) \le l_v$, and then define sets $\mathbb{D}_i$ as 
$ \mathbb{D}_i = \{x \in K \colon x \in D_{v}^{(i(v))} \text{ for all } v \in S\}.$
Under these conditions, the following properties hold:
\begin{itemize}
  \item[$(a)$] 
  $K$ is partitioned into a disjoint union in the form of $K = \bigcup_{i} \mathbb{D}_i$.
  \item[$(b)$] On $\mathbb{D}_i$, $\widehat{H}_{\phi}(x)/H_K(x)$ takes a constant value $\rho_i$.
\end{itemize}
Then, $Z_K(\phi, s)$ can be decomposed as follows.
\[ Z_{K}(\phi,s) = \sum_{x \in K} \widehat{H}_{\phi}(x)^{-s} = \sum_{i} \sum_{x \in \mathbb{D}_i} \widehat{H}_{\phi}(x)^{-s} = \sum_{i} \rho_i^{-s}\sum_{x \in \mathbb{D}_i} H_{K}(x)^{-s}.\]
Therefore, the study of $Z_K(\phi, s)$ is reduced to the investigation of the following partial height zeta function.

\begin{dfn}
  For a subset $\mathbb{D}$ of $K$, we define $W(\mathbb{D},s)$ as follows.
  \[ W(\mathbb{D},s) = \sum_{x\in \mathbb{D}} H_K(x)^{-s}. \]
\end{dfn}

Let $S$ be a finite subset of $M_K$, and for each $v \in S$, let $n_v$ be an integer and $c_v$ be an element in $K_v$. 
Then, consider $\mathbb{D}$ as a subset of $K$ in the following form.
\[\mathbb{D} = \{ x \in K \colon v(x - c_v) \ge n_v \text{ for all } v \in S\}.\]
Hsia computed the partial height zeta function for $\mathbb{D}$ in this form. 
To state the claims, let's introduce some symbols.

First, define $\mathbb{D} \subset K$ and $\mathbb{A}_{\mathbb{D}} \subset \mathbb{A}_{K}$ as follows:
\begin{align*}
  \mathbb{D} &= \{ x \in K \colon v(x - c_v) \ge n_v \text{ for all } v \in S\}, \\
  \mathbb{A}_{\mathbb{D}} &= \{ x = (x_v)_v \in \mathbb{A}_K \colon v(x_v - c_v) \ge n_v \text{ for all } v \in S\}.
\end{align*}
Next, define the divisors $E, E^{0}, E^{\infty}, \mathscr{E}, \mathscr{E}'$ as follows:
\begin{itemize}
  \item $E = \sum_{v \in S} n_v \cdot v$.
  \item $\sum_{v \in S} \left(\inf_{x \in \mathbb{D}} v(x) \right) \cdot v = E^{0} - E^{\infty}$, and $E^{0}, E^{\infty} \ge 0$ with $\mathrm{Supp}(E^{0}) \cap \mathrm{Supp}(E^{0}) = \varnothing$.
  \item $E^{\infty} = \mathscr{E} + \mathscr{E}'$, $\mathrm{Supp}(\mathscr{E}') = \{v \in \mathrm{Supp}(E^{\infty}) \colon v(E^{\infty}) = -n_v\}, \mathrm{Supp}(\mathscr{E}) = \{v \in \mathrm{Supp}(E^{\infty}) \colon v(E^{\infty}) > -n_v\}$
\end{itemize}

\begin{thm}[{\cite[Theorem 2.1.]{Hsia}}]\label{Hsia_partialheightzeta}
  Let $S \subset M_K$ be a finite subset of $M_K$. For $n_v \in \mathbb{Z}$ and $c_v \in K_v$ for all $v \in S$, we define    
  \begin{align*}
  \mathbb{D} &= \{ x \in K \colon v(x - c_v) \ge n_v \text{ for all } v \in S\}, \\
  \mathbb{A}_{\mathbb{D}} &= \{ x = (x_v)_v \in \mathbb{A}_K \colon v(x_v - c_v) \ge n_v \text{ for all } v \in S\}.
  \end{align*}
  Then, 
  \[ W(\mathbb{D}, s) = \int_{\mathbb{A}_{\mathbb{D}}} H_K(x)^{-s} d\mu (x) + \frac{B(q^{-s})}{\zeta_{K}(s)\prod_{v \in S}(1 - q_v^{-s})},\]
  where
  \[B(t) = t^{\deg \mathscr{E}} \sum_{0 \le E_1 \le \mathscr{E}} t^{\deg E_1} \prod_{v \in \mathrm{Supp}(\mathscr{E}-E_1)} (1-t^{f_v}) \sum_{\deg(D + E_1) < \eta} \delta(D + E_1) t^{\deg D},\]
  \[\delta(D) = \# \{x \in \mathbb{D} \colon (x)_{\infty} \le D + \mathscr{E}\} - q^{1-g+\deg (E+\mathscr{E}')}t^{\deg D}, \text{ and }\]
  \[ \eta = \deg(E + \mathscr{E}') + 2g - 1.\] 
\end{thm}

%\begin{exm}
%  If $K = \mathbb{F}_q(t)$ and $\phi(z) = z(z-t^{-1})$ then $\phi$ doesn't satisfy the condition.
%  To see this, we need to caluclate the local canonical height as the following way. 
%  This discussion is related to the computation of the non-Archimedean Julia set of $\phi$ (See Benedetto or Silverman). 

%  Let $v = \mathrm{ord}_t$ be the normalized $t$-adic additive valuation. 
%  We will show that $\widehat{\lambda}_{\phi, v} - \lambda_v$ takes infinitely many values.

%  Define two closed disks in $K_v=\mathbb{F}_q((t))$ as 
%\end{exm}

\section{Explicit computation of the dynamical height zeta function}

In this section, we derive an explicit formula for the dynamical height zeta function when the genus of $K$ is 0 or 1 and the dynamical system satisfies Condition \ref{Hsia_condition}
Basically, the calculations in this section are a refinement of Hsia's argument.
In particular, the central discussion is a more explicit computation of Theorem \ref{Hsia_partialheightzeta}.
Throughout this section, we consider $K$ as a function field over $\mathbb{F}_q$ with genus 0 or 1.

\begin{lem}
  Let $U \subset M_K$ be a finite subset of $M_K$. Let $\mathbb{D}(U) \subset K$ be a subset of $K$ defined by 
  $ \mathbb{D}(U) = \{ x \in K \colon v(x) \ge 0 \text{ for all } v \in U\}$
  and let $\mathbb{A}(U) \subset \mathbb{A}_K$ be a subset of $\mathbb{A}_K$ defined by
  $ \mathbb{A}(U) = \{ x = (x_v)_v \in \mathbb{A}_K \colon v(x_v) \ge 0 \text{ for all } v \in U\}$
  Then, 
  \[ W(\mathbb{D}(U), s) = \int_{\mathbb{A}(U)}H_K(x)^{-s} d\mu(x) + \frac{c(g)}{\zeta_K(s)\prod_{v \in U} (1 - q_v^{-s})}\]
\end{lem}

\begin{proof}
  Using the notation from Section \ref{Hsia_method}, we obtain
  \[ E = E^{0} = E^{\infty} = \mathscr{E} = \mathscr{E}' = 0\]
  for this situation. By theorem \ref{Hsia_partialheightzeta},
  \[ W(\mathbb{D}(U), s) = \int_{\mathbb{A}(U)}H_K(x)^{-s} d\mu(x) + \frac{B(q^{-s})}{\zeta_K(s)\prod_{v \in U} (1 - q_v^{-s})}, \]
  where 
  \begin{align*}
    &B(t)\\
    =& t^{\deg \mathscr{E}} \sum_{0 \le E_1 \le \mathscr{E}} t^{\deg E_1} \prod_{v \in \mathrm{Supp}(\mathscr{E}-E_1)} (1-t^{f_v}) \sum_{\deg(D + E_1) < \eta} \delta(D + E_1) t^{\deg D}\\
    =& \sum_{\deg D < 2g - 1} \delta(D) t^{\deg D}.
  \end{align*}
  If $g = 0$, then $B(t)$ is an empty sum and equals $0$. If $g = 1$, then 
  \[B(t) = \delta(0) = \# \{ x \in \mathbb{D}(U) \colon (x)_{\infty} = 0\} - q^{1-g} = \# \mathbb{F}_q - 1 = q - 1.\]
  Therefore, $B(t) = c(g)$ holds true.
\end{proof}

\begin{lem}\label{zeta_on_affinoids}
  Let $S \subset M_K$ be a finite subset of $M_K$. 
  For a subset $T \subset S$, we define a subset $\mathbb{D}_T \subset K$ and a subset $\mathbb{A}_T \subset\mathbb{A}_K$ as
  \[ \mathbb{D}_T = \{ x \in K \colon v(x) \ge 0 \text{ for all } v \in T, v(x) < 0 \text{ for all } v \in S \setminus T\},\]
  \[ \mathbb{A}_T = \{ x = (x_v)_v \in \mathbb{A}_K \colon v(x_v) \ge 0 \text{ for all } v \in T, v(x_v) < 0 \text{ for all } v \in S \setminus T\},\]
  respectively. Then, 
  \[ W(\mathbb{D}_T, s) = \int_{\mathbb{A}_T} H_K(x)^{-s} d\mu(x) + \frac{c(g)\prod_{v \in S \setminus T} (-q_v^{-s})}{\zeta_K(s) \prod_{v \in S} (1 - q_v^{-s})}\]
\end{lem}

\begin{proof}
  By the inclusion-exclusion principle, we obtain
  \[ \mathbb{D}_T = \bigcup_{U \subset S \setminus T} (-1)^{\# U} \mathbb{D}(T \cup U ), \mathbb{A}_T = \bigcup_{U \subset S \setminus T} (-1)^{\# U} \mathbb{A}(T \cup U ) .\]
  \begin{align*}
    & W(\mathbb{D}_T, s)\\
    =& \sum_{U \subset S \setminus T} (-1)^{\# U} W(\mathbb{D}(T \cup U), s)\\
    =& \sum_{U \subset S \setminus T} (-1)^{\# U} \left( \int_{\mathbb{A}(T \cup U)} H_K(x)^{-s} d\mu(x) + \frac{c(g)}{\zeta_K(s) \prod_{v \in T \cup U} (1 - q_v^{-s})} \right)\\
    =& \int_{\mathbb{A}_T} H_K(x)^{-s} d\mu(x) + \sum_{U \subset S \setminus T} (-1)^{\# U}\frac{c(g)}{\zeta_K(s) \prod_{v \in T \cup U} (1 - q_v^{-s})}
  \end{align*}
  Therefore, 
  \begin{align*}
    & W(\mathbb{D}_T, s) - \int_{\mathbb{A}_T} H_K(x)^{-s} d\mu(x)\\
    =& \frac{c(g)}{\zeta_K(s) \prod_{v \in T}(1 - q_v^{-s})}\sum_{U \subset S \setminus T} (-1)^{\# U}\prod_{v \in U} (1 - q_v^{-s})^{-1}\\
    =& \frac{c(g)}{\zeta_K(s) \prod_{v \in T}(1 - q_v^{-s})}\prod_{v \in S \setminus T} (1 - (1 - q_v^{-s})^{-1})\\
    =& \frac{c(g)}{\zeta_K(s) \prod_{v \in T}(1 - q_v^{-s})}\prod_{v \in S \setminus T} \frac{-q_v^{-s}}{1 - q_v^{-s}}\\
    =& \frac{c(g)\prod_{v \in S \setminus T} (-q_v^{-s})}{\zeta_K(s) \prod_{v \in S} (1 - q_v^{-s})}
  \end{align*}
\end{proof}

In addition, we further assume the following:
\begin{itemize}
  \item An element $f \in \mathcal{O}_K$ satisfies $v(f) < d$ for each $v \in M_K$.
  \item $\phi(z) = z^d + f^{-1}$.
\end{itemize}
Furthermore, let $S = \{v \in M_K \colon v(f) > 0\} = \{v \in M_K \colon \phi \text{ has bad reduction at } v \}$.
In this case, as can be deduced from Example \ref{computation_local_height}, the following holds.

\begin{lem}\label{computation_of_integration}
  Under the above assumptions, we have:
  \[ \int_{\mathbb{A}_K}\widehat{H}_{\phi}(x)^{-s} d\mu(x) = q^{1-g}\frac{\zeta_K(s-1)}{\zeta_K(s)} \prod_{v \in S} \frac{u_v^{v(f)} + (q_v - 1)u_v^d - q_v u_v^{d + v(f)}}{1 - u_v^d}\]
  where $u_v = q_v^{-s/d}$.
\end{lem}

\begin{proof}
  By Theorem \ref{Hsia_main},
  \[ \int_{\mathbb{A}_K}\widehat{H}_{\phi}(x)^{-s} d\mu(x) = q^{1-g} \frac{\zeta_{K}(s-1)}{\zeta_{K}(s)}\prod_{v \in S} \frac{\int_{K_v} \widehat{H}_{\phi, v} (x)^{-s} d \mu_v(x)}{\int_{K_v} H_{v} (x)^{-s} d \mu_v(x)}.\]
  Therefore, it suffices to compute the integrals for each $v \in S$.
  \[\int_{K_v} H_{v} (x)^{-s} d \mu_v(x) = \frac{1-q_v^{-s}}{1-q_v^{1-s}}\]
  and, as seen in Example 3.2,
  \begin{align*}
  \widehat{H}_{\phi,v}(x) = 
    \begin{cases}
      |x|_v & \text{ if } |x|_v > 1\\
      q_v^{v(f)/d} & \text{ if } |x|_v \le 1
    \end{cases}
  \end{align*}
  so that
  \[\int_{K_v} \widehat{H}_{\phi,v} (x)^{-s} d \mu_v(x) = q_v^{-sv(f)/d} + \frac{(q_v-1)q_v^{-s}}{1-q_v^{1-s}}.\]
  Therefore, 
  \[\frac{\int_{K_v} \widehat{H}_{\phi, v} (x)^{-s} d \mu_v(x)}{\int_{K_v} H_{v} (x)^{-s} d \mu_v(x)} = \frac{u_v^{v(f)} + (q_v - 1)u_v^d - q_v u_v^{d + v(f)}}{1 - u_v^d}.\]
\end{proof}

Based on the above discussion, we are now prepared to explicitly compute the dynamical height zeta function.

\begin{thm}\label{explicit_formula_genus_1}
  Let $K$ be a funciton field of a curve over a fintie field $\mathbb{F}_q$ of genus $0$ or $1$ and let $d \ge 2$ be an integer greater than or equal to $2$.
  Assume that $f \in \mathcal{O}_K$ satisfies $v(f) < d$ for all $v \in M_K$ and put $S = \{v \in M_K \colon v(f) > 0\}$. 
  Then, the dynamical height zeta function of $\phi(z) =z^d + f^{-1}$ is given by the next formula.
  \[Z_K(\phi, s) = q^{1-g}\frac{\zeta_K(s-1)}{\zeta_K(s)} \prod_{v \in S} \frac{u_v^{v(f)} + (q_v - 1)u_v^d - q_v u_v^{d + v(f)}}{1 - u_v^d} + \frac{c(g)}{\zeta_K(s)} \prod_{v \in S} \frac{u_v^{v(f)} - u_v^d}{1 - u_v^d}.\]
  Here, $\zeta_K$ represents the Dedekind zeta function of $K$, $q_v = q^{f_v}$ stands for the cardinality of the residue field at $v \in M_K$, $u_v = q_v^{-s/d}$, $c(0)=0$, and $c(1) = q-1$.
\end{thm}

\begin{proof}
  Let $S = \{v \in M_K \colon v(f) > 0 \} = \{v \in M_K \colon \phi \text{ has bad reduction at } v \}$. 
  By the computation in Example \ref{computation_local_height}, 
  \[ x \in \mathbb{D}_T \Rightarrow \frac{\widehat{H}_{\phi}(x)}{H_K(x)} = \rho_T := \prod_{v \in T} q_v^{v(f)/d} \]

  Since $K = \coprod_{T \subset S} \mathbb{D}_T$ and the above fact,
  \[Z_K(\phi, s)=\sum_{x \in K} \widehat{H}_{\phi}(x)^{-s}=\sum_{T \subset S} \sum_{x \in \mathbb{D}_T} \widehat{H}_{\phi}(x)^{-s}=\sum_{T \subset S} \rho_T^{-s} W(\mathbb{D}_T,s).\]
  By aplying Lemma \ref{zeta_on_affinoids},
  \begin{align*}
    &Z_K(\phi, s)\\ 
    =& \sum_{T \subset S} \rho_T^{-s} W(\mathbb{D}_T,s)\\
    =& \sum_{T \subset S} \rho_T^{-s} \left( \int_{\mathbb{A}_T} H_K(x)^{-s} d\mu(x) + \frac{c(g) \prod_{v \in S \setminus T} (-q_v^{-s})}{\zeta_K(s) \prod_{v \in S} (1 - q_v^{-s})} \right)\\
    =& \sum_{T \subset S} \int_{\mathbb{A}_T} (\rho_T H_K(x))^{-s} d\mu(x) + \frac{c(g)}{\zeta_K(s) \prod_{v \in S} (1 - q_v^{-s})}\sum_{T \subset S} \rho_T^{-s} \prod_{v \in S \setminus T} (-q_v^{-s})\\
    =& \int_{\mathbb{A}_K} \widehat{H}_\phi(x)^{-s} d \mu(x) + \frac{c(g)}{\zeta_K(s) \prod_{v \in S} (1 - q_v^{-s})}\sum_{T \subset S} \rho_T^{-s} \prod_{v \in S \setminus T} (-q_v^{-s})
  \end{align*}
  On the other hand, we have
  \begin{align*}
  \sum_{T \subset S} \rho_T^{-s} \prod_{v \in S \setminus T} (-q_v^{-s})
  &= \sum_{T \subset S} \left(\prod_{v \in T} q_v^{-sv(f)/d}\right)\left(\prod_{v \in S \setminus T} (-q_v^{-s})\right) \\
  &= \sum_{T \subset S} \left(\prod_{v \in T} u_v^{v(f)} \right)\left( \prod_{v \in S \setminus T} (-u_v^d)\right)\\
  &= \prod_{v \in S} (u_v^{v(f)} - u_v^{d}).
  \end{align*}
  By conbining the above discussion and Lemma \ref{computation_of_integration}, we obtain the desired result
\end{proof}

\section{A general discussion on the counting problem}

%The purpose of this section is giving a general discussion to describe the asymptotic behabior of the number of points with bounded height via the height zeta function in a general senario.
%We often use Tauberian theorems or other theorems such as Perron's formula to obtain asymptotic formula from zeta functions.
%However, those methods doesn't work well in our situation because our zeta function $Z_K(\phi, s)$ can have infinitely many poles in a vertical line $\mathrm{Re}(s) = \sigma$.
%So we will take another way as described in this section.
The purpose of this section is to provide a general discussion describing the asymptotic behavior of the number of points with bounded height, using the height zeta function in a general scenario.
In order to obtain asymptotic formulas from zeta functions, we often employ Tauberian theorems or other methods such as Perron's formula. 
However, these methods do not work well in our particular situation, as our zeta function $Z_K(\phi, s)$ may have infinitely many poles on a vertical line $\mathrm{Re}(s) = \sigma$. 
Therefore, we adopt an alternative approach as outlined in this section.

\subsection{An asymptotic formula}

In this subsection, we present a general statement in analytic number theory that will be utilized in the subsequent subsection.

First, we introduce a lemma that involves Stirling numbers of the second kind. 
We denote the Stirling numbers of the second kind as $\left \{ {n \atop k} \right \}$ and the $n$-th Bernoulli number as $B_n$. 
Stirling numbers of the second kind can be defined by the recurrence relation
\begin{align}\label{StirlingNumber}
   \left \{ {n + 1 \atop k} \right \} = k \left \{ {n \atop k} \right \} + \left \{ {n \atop k - 1} \right \},
\end{align}
along with the initial conditions
\begin{align}\label{InitialCondition}
   \left \{ {n \atop n} \right \} = 1 \text{ for all } n \ge 0, \quad \left \{ {n \atop 0} \right \} = \left \{ {0 \atop n} \right \} = 0 \text{ for all } n > 0.
\end{align}
The $n$-th Bernoulli number $B_n$ is defined by utilizing the $n$-th Taylor coefficient of $\frac{x}{e^x - 1}$ as follows:
\begin{align}\label{BernoulliNumber}
   \frac{x}{e^x - 1} = \sum_{n = 0}^{\infty} B_n \frac{x^n}{n!}.
\end{align}

\begin{lem}\label{TaylorExpansion}
  For all positive integers $n > 0$, the following equation holds.
  \[\sum_{k=1}^{n} \frac{(-1)^{n+k} \frac{(k-1)!}{(n-1)!}\left \{ {n \atop k} \right \}}{(1 - e^{-x})^k} = \frac{1}{x^n} - \frac{1}{(n-1)!}\sum_{m=0}^{\infty} \frac{B_{m+n}}{m+n} \frac{(-x)^m}{m!}  \]
\end{lem}

\begin{proof}
  We can easily see that the statement of the lemma is equivalent to the following equation.
  \[ \sum_{k=1}^{n} \frac{(k - 1)! \left \{ {n \atop k} \right \}}{(e^x - 1)^k} = \frac{(n-1)!}{x^n} - (-1)^n \sum_{m=0}^{\infty} \frac{B_{n+m}}{n+m} \frac{x^m}{m!}.\]
  We will prove this formula by mathematical induction on $n$. We put
  \begin{align*}
    F_n(x) &= \sum_{k=1}^{n} \frac{(k - 1)! \left \{ {n \atop k} \right \}}{(e^x - 1)^k},\\
    G_n(x) &= \frac{(n-1)!}{x^n} - (-1)^n \sum_{m=0}^{\infty} \frac{B_{n+m}}{n+m} \frac{x^m}{m!}.
  \end{align*} 
  The condition $F_1(x) = G_1(x)$ is equivalent to the definition of Bernoulli numbers (\ref{BernoulliNumber}). 
  Clearly $G_n$ satisfies the recrrence relation $\frac{d}{dx} G_n(x) = -G_{n+1}(x)$. 
  By using the recrrence relation of stirling numbers (\ref{StirlingNumber}), we can verify $\frac{d}{dx} F_n(x) = -F_{n+1}(x)$.
  So we can check the relation $F_n(x) = G_n(x)$ inductively.   
\end{proof}

\begin{thm}\label{GeneralAsymptotic}
  Let $\alpha > 0$ be a positive real number and $\{a_n\}_n$ be a sequence of complex numbers.
  Assume that the function $f$ defined by 
  \[ f(s) = \sum_{m=0}^{\infty} a_m \alpha^{-ms}\]
  converges for $\mathrm{Re}(s) > \sigma_0$ for some $\sigma_0 \in \mathbb{R}_{>0}$ and continued meromorphically to the half plane $\{z \in \mathbb{C} \mid \mathrm{Re}(s) > - \sigma_1\}$ for some $\sigma_1 \in \mathbb{R}_{>0}$.
  Then the following asymptotic formula holds.
  \[ \sum_{m \le T}a_m = \sum_{m \le T} \sum_{a \in \mathrm{Poles}} \alpha^{am} \sum_{n = 1}^{N_a} c_{n}(a) \frac{m^{n-1}}{(n - 1)!} + O(1)\]
  as $T \to \infty$. Where
  \[ \mathrm{Poles} = \left \{a \in \mathbb{C} \colon a \text{ is a pole of } f, 0 \le \mathrm{Re}(a), 0 \le \mathrm{Im}(a) < \frac{2\pi}{\log \alpha} \right \}\]
  is a set of poles of $f$ and
  \[ f(s) = \sum_{n=1}^{N_a} \frac{c_n(a)}{(\log \alpha)^n(s - a)^n} + (\text{holomorphic at }s=a)\] 
  is the Laurent expansion of $f$ at $s = a$.
\end{thm}

\begin{proof}
  First, $f(s)$ is a periodic funtion with period $2\pi\sqrt{-1}/\log\alpha$ by its definition. So notice that $\mathrm{Poles}$ gives a system of representatives of the set of poles  of $f$ with non-negative real part modulo $\frac{2\pi\sqrt{-1}}{\log \alpha}\mathbb{Z}$. 
  By Lemma \ref{TaylorExpansion},
  \[ \sum_{k=1}^{n} \frac{(-1)^{n+k}\frac{(k-1)!}{(n-1)!}\left \{ {n \atop k} \right \}}{(1 - \alpha^{-(s-a)})^k} - \frac{1}{(\log \alpha)^n (s - a)^n} \]
  is holomorphic at $s = a$. Therefore, for $a \in \mathrm{Poles}$, 
  \[ f(s) - \sum_{n=1}^{N_a} c_n(a) \sum_{k=1}^{n} \frac{(-1)^{n+k}\frac{(k-1)!}{(n-1)!}\left \{ {n \atop k} \right \}}{(1 - \alpha^{-(s-a)})^k}\]
  is holomorphic at $s = a$.  Since $\alpha^{-(s-a)}$ is the periodic function with period $2\pi\sqrt{-1}/\log\alpha$, we can see that
  \[ f(s) - \sum_{a \in \mathrm{Poles}} \sum_{n=1}^{N_a} c_n(a) \sum_{k=1}^{n} \frac{(-1)^{n+k}\frac{(k-1)!}{(n-1)!}\left \{ {n \atop k} \right \}}{(1 - \alpha^{-(s-a)})^k}\]
  is a holomorphic function on $\mathrm{Re}(s) > -\sigma_2$ for some positive real number $0 < \sigma_2 < \sigma_1$.
  On the other hand, we can expand the sum along $\mathrm{Poles}$ as a power series of $\alpha^{-s}$ as follows.
  \begin{align*}
    &\sum_{a \in \mathrm{Poles}} \sum_{n=1}^{N_a} c_n(a) \sum_{k=1}^{n} \frac{(-1)^{n+k}\frac{(k-1)!}{(n-1)!}\left \{ {n \atop k} \right \}}{(1 - \alpha^{-(s-a)})^k} \\
    =& \sum_{a \in \mathrm{Poles}} \sum_{n=1}^{N_a} c_n(a) \sum_{k=1}^{n} (-1)^{n+k}\frac{(k-1)!}{(n-1)!}\left \{ {n \atop k} \right \} \sum_{m=0}^{\infty} \binom{-k}{m} (-\alpha^{-(s-a)})^m \\
    =& \sum_{a \in \mathrm{Poles}} \sum_{n=1}^{N_a} c_n(a) \sum_{k=1}^{n} (-1)^{n+k}\frac{(k-1)!}{(n-1)!}\left \{ {n \atop k} \right \} \sum_{m=0}^{\infty} \binom{m + k - 1}{m} \alpha^{ma} \alpha^{-ms} \\
    =& \sum_{m=0}^{\infty} \alpha^{-ms} \sum_{a \in \mathrm{Poles}} \alpha^{am} \sum_{n=1}^{N_a} \frac{c_n(a)}{(n-1)!} \sum_{k=1}^{n} (-1)^{n-k} \left\{ {n \atop k} \right\} (k-1)! \binom{m + k - 1}{k - 1}
  \end{align*}
  The first equality is the binomial theorem. 
  The second identity is given by the formula of binomial coefficient, that is, $\binom{-k}{m} = (-1)^m \binom{m+k-1}{m}$.
  The part concerning the summation with respect to $k$ can be expressed as follows using the Pochhammer symbol $(x)_n = x (x+1) \cdots (x+n-1)$.
  \begin{align*}
    &\sum_{k=1}^{n} (-1)^{n-k} \left\{ {n \atop k} \right\} (k-1)! \binom{m + k - 1}{k - 1}\\
    =& \frac{1}{m}\sum_{k=1}^{n} (-1)^{n-k} \left\{ {n \atop k} \right\} (m)_k\\
    =& m^{n-1}
  \end{align*}
  The first equality comes from the fact $x^n = \sum_{k=1}^{n} (-1)^{n-k}\left \{ {n \atop k} \right \} (x)_k$.
  Therefore,
  \begin{align*}
    &\sum_{a \in \mathrm{Poles}} \sum_{n=1}^{N_a} c_n(a) \sum_{k=1}^{n} \frac{(-1)^{n+k}\frac{(k-1)!}{(n-1)!}\left \{ {n \atop k} \right \}}{(1 - \alpha^{-(s-a)})^k}\\
    =& \sum_{m=0}^{\infty} \alpha^{-ms} \sum_{a \in \mathrm{Poles}} \alpha^{am} \sum_{n=1}^{N_a} c_n(a) \frac{m^{n-1}}{(n-1)!}.
  \end{align*}

  The power series of $\alpha^{-s}$
  \[\sum_{m=0}^{\infty} \left( a_m - \sum_{a \in \mathrm{Poles}} \alpha^{ma} \sum_{n=1}^{N_a} c_n(a) \frac{m^{n-1}}{(n-1)!} \right) \alpha^{-ms}\]
  is holomorphic for $\mathrm{Re}(s) > - \sigma_2$. So the series given by substisuting $s = -\sigma_2/2$ into the above function converges absolutely. 
  Especially, we obtain a constant $C > 0$ such that for all $m \in \mathbb{Z}_{\ge 0}$, 
  \[ \left | a_m - \sum_{a \in \mathrm{Poles}} \alpha^{ma} \sum_{n=1}^{N_a} c_n(a) \frac{m^{n-1}}{(n-1)!} \right | \alpha^{m\sigma_2/2} \le C .\]
  Therefore,
  \begin{align*}
    \left | \sum_{m \le T}a_m - \sum_{m \le T} \sum_{a \in \mathrm{Poles}} \alpha^{am} \sum_{n = 1}^{N_a} c_{n}(a) \frac{m^{n-1}}{(n - 1)!} \right | \le C \sum_{n \le T} (\alpha^{-\sigma_2/2})^m < \frac{C}{1-\alpha^{-\sigma_2/2}}.
  \end{align*}
\end{proof}

\subsection{Asymptotic behabior of the number of points with bounded height}\label{AsymptoticHeightGeneral}

In this subsection, we assume the following assumptions.

\begin{itemize}
  \item $S$ is a set.
  \item $H \colon S \to \mathbb{R}_{>0}$ is a map from $S$ to the set of positive real numbers $\mathbb{R}_{>0} = \{x \in \mathbb{R} \mid x > 0\}$.
  \item For a positive real number $B >0 $, $N(S, H, B) = \# \{ x \in S \colon H(x) \le B\}$.
  \item Assume that $N(S, H, B) < \infty$ for each positive real number $B > 0$.
  \item The zeta function $Z(S, H, s)$ is defined by the folloing Dirichlet series:
    \[ Z(S, H, s) = \sum_{x \in S} \frac{1}{H(x)^s}\]
  \item Assume that the above series converges absolutely for $\mathrm{Re}(s) > \sigma_0$ and $Z(S, H, s) \in \mathbb{Q}(\alpha^{-s})$ for some $\sigma_0, \alpha \in \mathbb{R}_{>0}$.
\end{itemize}

\begin{thm}\label{zeta_to_asymptotic}
  Let $\mathrm{Poles}$ be a set of poles of $Z(S,H,s)$ defined as follows:
  \[ \mathrm{Poles} = \left \{a \in \mathbb{C} \colon a \text{ is a pole of } f, 0 \le \mathrm{Re}(a), 0 \le \mathrm{Im}(a) < \frac{2\pi}{\log \alpha} \right \}.\]
  For a pole $a \in \mathrm{Poles}$, we write the Laurent expansion at $s = a$ as
  \[ Z(S,H,s) = \sum_{n=1}^{N_a} \frac{c_n(a)}{(\log \alpha)^n(s - a)^n} + (\text{holomorphic at }s=a).\]
  Then, the following asymptotic formula for $N(S,H,B)$ holds.
  \begin{align*}
    N(S, H, B) =  \sum_{\alpha^m \le B} \sum_{a \in \mathrm{Poles}} \alpha^{am} \sum_{n = 1}^{N_a} c_{n}(a) \frac{m^{n-1}}{(n - 1)!} + O(1).
  \end{align*}
\end{thm}

\begin{proof}
  Since $Z(S,H,s) \in \mathbb{Q}$, we have
  \[ Z(S,H,s) = \sum_{m=0}^{\infty} a_m \alpha^{-ms} \]
  On the other hand, by $Z(S,H,s) = \sum_{x \in S}H(x)^{-s}$ and the uniqueness of coefficients of Dirichlet series, we have $a_n = \#\{ x \in S \colon H(x) = \alpha^m\}$.
  Therefore, 
  \[ N(S,H,B) = \sum_{\alpha^m \le B} a_m. \]
  So the desired formula is a direct application of Theorem \ref{GeneralAsymptotic}.
\end{proof}

\section{Examples}\label{Example}
In this section, we utilize Theorem \ref{explicit_formula_genus_1} and Theorem \ref{zeta_to_asymptotic} to derive explicit asymptotic formulas for $N(K, \widehat{H}_{\phi},B)$, which are obtained by applying these theorems to specific situations.

\begin{exm}
  Let $K = \mathbb{F}_5(t, \sqrt{t^3+3})$ and $\phi(z) = z^2 + 1/t \in K[z]$. Let's derive an asymptotic formula for $N(K,\phi, B)$.
  First, we list basic facts about the field $K$.
  \begin{itemize}
    \item The field $K$ is the function field of the elliptic curve $y^2 = x^3 + 3$ over $\mathbb{F}_5$.
    \item Since $\left(\frac{f(t)}{t}\right)=\left(\frac{f(0)}{t}\right)=f(0)^{(5-1)/2}=-1$ in $\mathbb{F}_5$, the prime ideal $t\mathbb{F}_5[t]$ in $\mathbb{F}_5[t]$ is inert in the field $K$. Here, $\left(\frac{\,\,}{\,\,}\right)$ is the Jacobi symbol.
    \item Therefore, $S = \{v \in M_K \colon \phi \text{ has bad reduction at } v\} = \{v\}$, where $v$ is the place corresponding to the prime ideal $t\mathcal{O}_K$. Furthermore, $q_v = 5^2 = 25, v(t) = 1$. 
    \item Since $\{(x,y) \in \mathbb{F}_5^2 \colon y^2 = x^3 + 3\} = \{(1,2),(1,3),(2,1),(2,4),(3,0)\}$ has $5$ points, so the Frobenius trace of the elliptic curve $y^2=x^3+3$ is $0$ and the zeta function of $K$ is 
      \[ \zeta_K(s) = \frac{1+5^{1-2s}}{(1-5^{-s})(1-5^{1-s})}. \]
  \end{itemize} 
  Applying the above facts to Theorem \ref{explicit_formula_genus_1}, the zeta function $Z_K(\phi,s)$ can be calculated as follows.
  \begin{align*}
    &Z_K(\phi,s)\\
    =&\frac{\zeta_K(s-1)}{\zeta_K(s)} \frac{u_v + (q_v - 1)u_v^2 - q_v u_v^{3}}{1 - u_v^2} + \frac{q-1}{\zeta_K(s)} \frac{u_v - u_v^2}{1 - u_v^2}\\
    =&5^{-s} \frac{(1-5^{-s})(1+5^{2-s})(1+5^{3-2s})}{(1+5^{-s})(1-5^{2-s})(1+5^{1-2s})} + 4 \cdot 5^{-s}\frac{(1-5^{-s})(1+5^{1-s})}{(1+5^{-s})(1+5^{1-2s})}.
  \end{align*}
  In accordance with the notation used in Theorem \ref{zeta_to_asymptotic}, 
  \[\mathrm{Poles} = \left\{2, \frac{1}{2} + \frac{\pi \sqrt{-1}}{2\log(5)}, \frac{1}{2} + \frac{3 \pi \sqrt{-1}}{2\log(5)}, \frac{\pi \sqrt{-1}}{2\log(5)} \right\}.\]
  Each pole within the set $\mathrm{Poles}$ is simple, and the Laurent coefficient at each pole $a \in \mathrm{Poles}$ is provided in Table \ref{Laurent_coefficients_case1}.
  \begin{table}[htb]
    \caption{Laurent coefficients}\label{Laurent_coefficients_case1}
    \centering
    {\renewcommand{\arraystretch}{2.5}
    \begin{tabular}{c|cccc}
      $a \in \mathrm{Poles}$  & $2$             & $\dfrac{1}{2} + \dfrac{\pi \sqrt{-1}}{2\log(5)}$  & $\dfrac{1}{2} + \dfrac{3\pi \sqrt{-1}}{2\log(5)}$ & $\dfrac{\pi \sqrt{-1}}{\log{5}}$ \\ \hline
      $c_1(a)$                & $\dfrac{8}{91}$ & $\dfrac{2}{7} (23 - 3\sqrt{-5})$                  & $\dfrac{2}{7} (23 + 3\sqrt{-5})$                  & $\dfrac{400}{13}$
    \end{tabular}
    }
  \end{table}

  In this situation, Theorem \ref{zeta_to_asymptotic} says 
  \begin{align*}
    &N(K,\widehat{H}_{\phi}, B) \\
    =& \sum_{5^m \le B} \sum_{a \in \mathrm{Poles}} 5^{am} c_{1}(a) + O(1)\\
    =& \sum_{5^m \le B} \left( \frac{8}{91} 25^m + \frac{2}{7} (23 - 3\sqrt{-5}) (\sqrt{-5})^m + \frac{2}{7} (23 + 3\sqrt{-5}) (-\sqrt{-5})^m + \dfrac{400}{13} (-1)^m \right) + O(1)\\
    =& \frac{8}{91}\sum_{5^m \le B} 25^m + \frac{92}{7} \sum_{5^{2m} \le B} (-5)^m + \frac{30}{7} \sum_{5^{2m+1} \le B} (-5)^m + O(1)\\
    =& \frac{25}{273} 25^{\left\lfloor \frac{\log B}{\log 5} \right\rfloor} + \frac{5}{21} \left(46 (-5)^{\left\lfloor \frac{1}{2} \cdot \frac{\log B}{\log 5} \right\rfloor} + 15 (-5)^{\left\lfloor \frac{1}{2} \left( \frac{\log B}{\log 5} - 1\right)\right\rfloor}\right) + O(1).
  \end{align*}
  Here, it should be noted that the first term has an order of $B^2\times O(1)$, and the second term has an order of $\sqrt{B}\times O(1)$. 
  This corresponds to the real part of the poles of the zeta function $Z_K(\phi, s)$.
\end{exm}

\begin{exm}
  Let $K = \mathbb{F}_5(t, \sqrt{t^3+1})$ and $\phi(z) = z^2 + 1/t \in K[z]$. Let's derive an asymptotic formula for $N(K,\phi, B)$.
  First, we list basic facts about the field $K$.
  \begin{itemize}
    \item The field $K$ is the function field of the elliptic curve $y^2 = x^3 + 1$ over $\mathbb{F}_5$.
    \item Since $\left(\frac{f(t)}{t}\right)=\left(\frac{f(0)}{t}\right)=f(0)^{(5-1)/2}=1$ in $\mathbb{F}_5$, the prime ideal $t\mathbb{F}_5[t]$ in $\mathbb{F}_5[t]$ is split in the field $K$. 
    \item Therefore, $S = \{v \in M_K \colon \phi \text{ has bad reduction at } v\} = \{v_1, v_2\}$, where $v_1, v_2$ are places corresponding to prime ideals lying over $t\mathbb{F}_5[t]$. Furthermore, $q_{v_1} = q_{v_2} = 5, v_1(t) = v_2(t) = 1$. 
    \item Since $\{(x,y) \in \mathbb{F}_5^2 \colon y^2 = x^3 + 1\} = \{(0,1),(0,4),(2,2),(2,3),(4,0)\}$ has $5$ points, so the Frobenius trace of the elliptic curve $y^2=x^3+1$ is $0$ and the zeta function of $K$ is 
      \[ \zeta_K(s) = \frac{1+5^{1-2s}}{(1-5^{-s})(1-5^{1-s})}. \]
  \end{itemize} 
  Applying the above facts to Theorem \ref{explicit_formula_genus_1}, the zeta function $Z_K(\phi,s)$ can be calculated as follows.
  \begin{align*}
    &Z_K(\phi,s)\\
    =&\frac{\zeta_K(s-1)}{\zeta_K(s)} \prod_{v =v_1,v_2}\frac{u_v + (q_v - 1)u_v^2 - q_v u_v^{3}}{1 - u_v^2} + \frac{q-1}{\zeta_K(s)} \prod_{v =v_1,v_2} \frac{u_v - u_v^2}{1 - u_v^2}\\
    =& u^2 \frac{(1-u)(1+5u)(1+125u^4)}{(1+u)(1-5u)(1+5u^4)} + 4u^2 \frac{(1-u)(1-5u^2)}{(1+u)(1+5u^4)},
  \end{align*}
  where $u = 5^{-s/2}$. Therefore,
  \[\mathrm{Poles} = \left\{2, \frac{2\pi\sqrt{-1}}{\log{5}}\right\}\cup \left\{\frac{1}{2} + \frac{k \pi \sqrt{-1}}{2\log{5}} \colon k \in \{1,3,5,7\} \right\}.\]
  Each pole in the set $\mathrm{Poles}$ is simple, and the Laurent coefficient at pole at each pole $a \in \mathrm{Poles}$ is provided in Table \ref{Laurent_coefficients_case1}.
  \begin{table}[htb]
    \caption{Laurent coefficients}\label{Laurent_coefficients_case2}
    \centering
    {\renewcommand{\arraystretch}{2.5}
    \begin{tabular}{c|ccc}
      $a \in \mathrm{Poles}$  & $2$             & $\dfrac{1}{2} + \dfrac{k \pi \sqrt{-1}}{2\log(5)}, k \in \{1,3,5,7\}$      & $\dfrac{2 \pi \sqrt{-1}}{\log{5}}$ \\ \hline
      $c_1(a)$                & $\dfrac{8}{63}$ & $\displaystyle \sum_{i=0}^{3} b_i 5^{i/4} \zeta_8^{ik}$ & $-\dfrac{200}{3}$
    \end{tabular}
    }
  \end{table}

  In Table \ref{Laurent_coefficients_case2}, $\zeta_8 = (1+\sqrt{-1})/\sqrt{2}$ is a primitive $8$-th root of unity and $(b_0,b_1,b_2,b_3) = \frac{2}{63} (-43,130,-37,20)$.
  By Theorem \ref{zeta_to_asymptotic}, we have
  \begin{align*}
    &N(K,\widehat{H}_{\phi}, B) \\
    =& \sum_{\sqrt{5}^m \le B} \sum_{a \in \mathrm{Poles}} 5^{am} c_{1}(a) + O(1)\\
    =& \sum_{\sqrt{5}^m \le B} \left( \frac{8}{63} 5^m + \sum_{k\in\{1,3,5,7\}}\sum_{i=0}^{3} b_i 5^{i/4} \zeta_8^{ik} (5^{1/4}\zeta_8^k)^m  - \dfrac{200}{3} (-1)^m \right) + O(1)\\
    =& \frac{8}{63}\sum_{\sqrt{5}^m \le B} 5^m + \frac{8}{63}\sum_{i=0}^{3}c_i \sum_{\sqrt{5}^{4m+i} \le B}(-5)^m +O(1)\\
    =& \frac{10}{63} 5^{\lfloor \frac{2\log{B}}{\log{5}} \rfloor} + \frac{20}{189} \sum_{i=0}^{3} c_i \cdot (-5)^{\left\lfloor \frac{1}{4} \left( \frac{2\log{B}}{\log{5}} - i \right) \right\rfloor} + O(1),
  \end{align*}
  where $(c_0,c_1,c_2,c_3) = (-43,-100,185,-650)$.
  Similarly to the previous example, it is important to note that, in correspondence with the real parts of the poles of the zeta function $Z_K(\phi,s)$, the first term has an order of $B^2\times O(1)$, and the second term has an order of $\sqrt{B}\times O(1)$.
\end{exm}

\bibliography{height_zeta_2023}
\bibliographystyle{plain}

\end{document}